\documentclass[a4paper,11pt]{article}
\usepackage[utf8]{inputenc}
\usepackage[T1]{fontenc}

\usepackage{amsthm}
\usepackage{tikz}
\usepackage{mathrsfs,amssymb,amsfonts}  
\usepackage{enumitem}
\usepackage{fullpage}
\usepackage{cite}
\usepackage{float}
\restylefloat{table}

\usepackage{thmtools}
\usepackage{mathtools, comment}
\usepackage{amssymb}
\usepackage[nomath]{lmodern}

\usepackage{graphicx}
\usepackage{pgf,tikz,tkz-graph,subcaption}
\usetikzlibrary{arrows,shapes}
\usetikzlibrary{decorations.pathreplacing}
\usepackage{tkz-berge}
\usepackage{enumitem}
\usepackage[normalem]{ulem}
\usepackage{hyperref}
\hypersetup{colorlinks = true, linkcolor = blue, citecolor = blue, urlcolor = blue}
\newcommand*{\ceilfrac}[2]{\mathopen{}\left\lceil\frac{#1}{#2}\right\rceil\mathclose{}}
\newcommand*{\floorfrac}[2]{\mathopen{}\left\lfloor\frac{#1}{#2}\right\rfloor\mathclose{}}

\newcommand*{\abs}[1]{\lvert #1\rvert}
\newcommand{\imax}{i_{\text{max}}}

\allowdisplaybreaks

\usepackage[margin=1in]{geometry}
\newtheorem{innercustomthm}{Theorem}
\newenvironment{customthm}[1]
  {\renewcommand\theinnercustomthm{#1}\innercustomthm}
  {\endinnercustomthm}

\newtheorem{defi}{Definition}
\newtheorem{conj}[defi]{Conjecture}
\newtheorem{cor}[defi]{Corollary}
\newtheorem{thr}[defi]{Theorem}
\newtheorem{lem}[defi]{Lemma}

\newtheorem{prop}[defi]{Proposition}

\newtheorem{question}[defi]{Question}
\newtheorem{remark}[defi]{Remark}

\newcommand*{\myproofname}{Proof}

\def\G{\mathcal{G}}
\def\M{\mathcal{M}}

\def\F{\mathcal{F}}
\def\E{\mathcal{E}}

\title{The minimum number of maximal independent sets in twin-free graphs}
\author{Stijn Cambie\thanks{Department of Computer Science, KU Leuven Campus Kulak-Kortrijk, 8500 Kortrijk, Belgium. Supported by the Institute for Basic Science (IBS-R029-C4), Internal Funds of KU Leuven (PDM fellowship PDMT1/22/005) and a postdoctoral fellowship by the Research Foundation Flanders (FWO) with grant number 1225224N.
E-mail: {\tt stijn.cambie@hotmail.com}} \and Stephan Wagner\thanks{Institute of Discrete Mathematics, TU Graz, 8010 Graz, Austria and Department of Mathematics, Uppsala University, 751 06 Uppsala, Sweden. Supported by the Knut and Alice Wallenberg Foundation (KAW 2017.0112) and the Swedish research council (VR), grant  2022-04030. E-mail: {\tt stephan.wagner@tugraz.at}} 
}

\date{}

\begin{document}

\maketitle

\begin{abstract}
    The problem of determining the maximum number of maximal independent sets in certain graph classes dates back to a paper of Miller and Muller and a question of Erd\H{o}s and Moser from the 1960s. The minimum was always considered to be less interesting due to simple examples such as stars. In this paper we show that the problem becomes interesting when restricted to twin-free graphs, where no two vertices have the same open neighbourhood. We consider the question for arbitrary graphs, bipartite graphs and trees. The minimum number of maximal independent sets turns out to be logarithmic in the number of vertices for arbitrary graphs, linear for bipartite graphs and exponential for trees. In the latter case, the minimum and the extremal graphs have been determined earlier by Taletski\u{\i} and Malyshev, but we present a shorter proof.
\end{abstract}

\section{Introduction}

\subsection{History on the maximum number of maximal independent sets}

Let $\imax(G)$ denote the number of maximal independent sets of a graph $G$, i.e., the number of independent sets that are not contained in any larger independent set. Answering a question on the number of (maximal) cliques posed by Erd\H{o}s and Moser, Moon and Moser~\cite{MM65}, independently from Miller and Muller~\cite{MillerMuller60} (see also~\cite{W11} for a short alternative proof) proved (by considering the complement) that for a graph $G$ of order $n$, where $n \ge 2$, $$\imax(G) \le 
\begin{cases}
 3^{\frac{n}{3}} & \text{ if } n\equiv 0 \pmod 3,\\
4\cdot 3^{\frac{n-4}{3}} & \text{ if } n\equiv 1 \pmod 3,\\ 
2\cdot 3^{\frac{n-2}{3}} & \text{ if } n\equiv 2 \pmod 3.\\ 
\end{cases}$$

The extremal graphs are disjoint unions of (at most two) $K_2$s and many $K_3$s.
In the original setting (for the number of maximal cliques), the extremal graph is a balanced $\ceilfrac{n}{3}-$partite graph.
Wilf~\cite{Wilf86} proved that among all trees of order $n$ the spiders (presented in Figure~\ref{fig:spiders}) maximize the number of maximal independent sets. For a tree $T$ with $n$ vertices,

$$\imax(T) \le 
\begin{cases}
 2^{\frac{n}{2}-1}+1 & \text{ if } n\equiv 0 \pmod 2,\\
2^{\frac{n-1}{2}} & \text{ if } n\equiv 1 \pmod 2.\\ 
\end{cases}$$

\begin{figure}[h]

\begin{minipage}[b]{.45\linewidth}
\begin{center}
    \begin{tikzpicture}
    {
	\foreach \x in {0,1,2,3,4}{\draw[thick] (2,2) -- (\x,1);}				
	\foreach \x in {0,1,2,3,4}{\draw[thick] (\x,1) -- (\x,0);}
	
    \foreach \x in {0,1,2,3,4}{\draw[fill] (\x,1) circle (0.1);}
	\foreach \x in {0,1,2,3,4}{\draw[fill] (\x,0) circle (0.1);}
    \draw[fill] (2,2) circle (0.1);
			
	}
	\end{tikzpicture}\\
\subcaption{An odd spider}
\label{fig:oddspider}
\end{center}
\end{minipage}\quad\begin{minipage}[b]{.45\linewidth}

\begin{center}
    \begin{tikzpicture}
    {
	\foreach \x in {0,1,2,3,4}{\draw[thick] (2,2) -- (\x,1);}				
	\foreach \x in {0,1,2,3,4}{\draw[thick] (\x,1) -- (\x,0);}
	\draw[thick] (2,2) -- (4,2);
    \foreach \x in {0,1,2,3,4}{\draw[fill] (\x,1) circle (0.1);}
	\foreach \x in {0,1,2,3,4}{\draw[fill] (\x,0) circle (0.1);}
    \draw[fill] (2,2) circle (0.1);
	\draw[fill] (4,2) circle (0.1);		
	}
	\end{tikzpicture}\\
\subcaption{An even spider}
\label{fig:evenspider}
\end{center}
\end{minipage}
\caption{Two spiders.}\label{fig:spiders}
\end{figure}
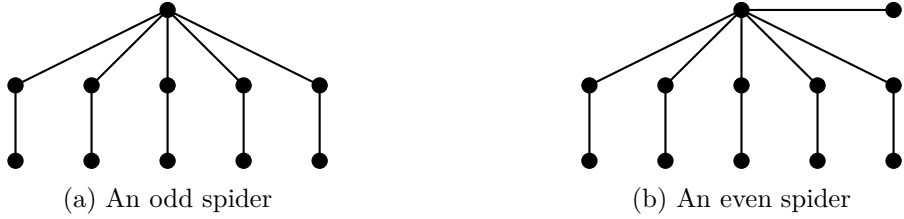

By induction, one can verify with Wilf's idea that the extremal tree is unique when $n$ is odd, but this is not the case when $n$ is even. Sagan~\cite{Sagan88} provided an alternative proof for Wilf's result and also characterized the extremal trees of even order as batons of length $1$ or $3$.
Here batons are subdivisions of trees of diameter at most $3$ in which the pendent edges are subdivided once and the central edge (or one edge of a star) is not subdivided or subdivided twice. Examples of such batons are presented in Figure~\ref{fig:batons}.
From this one can also conclude that there are precisely $\frac n2-1$ extremal trees if $n\ge 4$ is even.

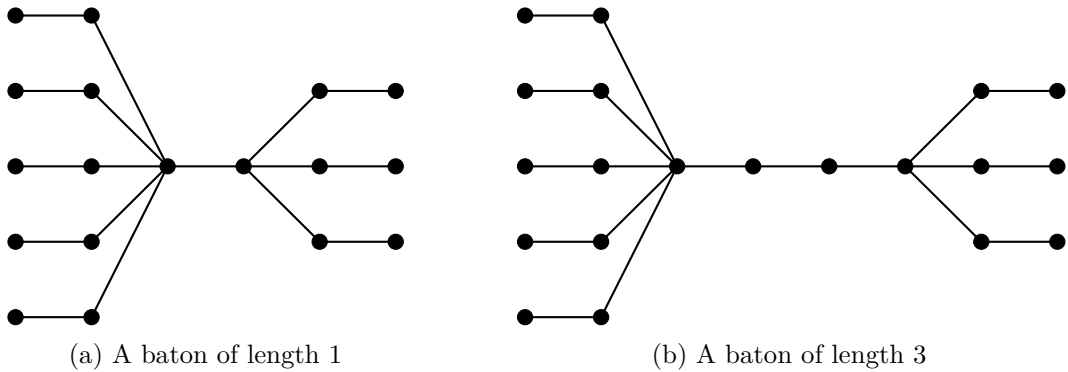
\begin{figure}[h]
\begin{minipage}[b]{.45\linewidth}
\begin{center}
    \begin{tikzpicture}
    {

	\foreach \x in {0,1,2,3,4}{\draw[thick] (2,2) -- (1,\x);}				
	\foreach \x in {0,1,2,3,4}{\draw[thick] (1,\x) -- (0,\x);}
	
    \foreach \x in {0,1,2,3,4}{\draw[fill] (1,\x) circle (0.1);}
	\foreach \x in {0,1,2,3,4}{\draw[fill] (0,\x) circle (0.1);}
    \draw[fill] (2,2) circle (0.1);
    \draw[fill] (3,2) circle (0.1);
    \draw[thick] (3,2) -- (2,2);
    
    \foreach \x in {1,2,3}{\draw[thick] (3,2) -- (4,\x);
    \draw[thick] (4,\x) -- (5,\x);
    \draw[fill] (4,\x) circle (0.1);
    \draw[fill] (5,\x) circle (0.1);
    }	
    }
	\end{tikzpicture}\\
\subcaption{A baton of length $1$}
\label{fig:baton1}
\end{center}
\end{minipage}
\quad\begin{minipage}[b]{.45\linewidth}

\begin{center}
\begin{tikzpicture}
    {
	\foreach \x in {0,1,2,3,4}{\draw[thick] (2,2) -- (1,\x);}				
	\foreach \x in {0,1,2,3,4}{\draw[thick] (1,\x) -- (0,\x);}
	
    \foreach \x in {0,1,2,3,4}{\draw[fill] (1,\x) circle (0.1);}
	\foreach \x in {0,1,2,3,4}{\draw[fill] (0,\x) circle (0.1);}
    \draw[fill] (2,2) circle (0.1);
    \foreach \y in {3,4,5}{
    \draw[fill] (\y,2) circle (0.1);
    \draw[thick] (\y,2) -- (\y-1,2);
    }
    
    \foreach \x in {1,2,3}{\draw[thick] (5,2) -- (6,\x);
    \draw[thick] (6,\x) -- (7,\x);
    \draw[fill] (6,\x) circle (0.1);
    \draw[fill] (7,\x) circle (0.1);
    }
			
	}
	\end{tikzpicture}\\
\subcaption{A baton of length $3$}
\label{fig:baton3}
\end{center}
\end{minipage}
\caption{Batons of length $1$ and $3$.}\label{fig:batons}
\end{figure}

Wilf also asked about the maximum of $\imax(G)$ when considering arbitrary connected graphs of given order. This question was answered independently by F\"{u}redi~\cite{Fur87} (for large $n$) and Griggs et al.~\cite{GGG88}.
They proved that for a connected graph $G$,
$$\imax(G) \le 
\begin{cases}
 2\cdot 3^{\frac{n}{3}-1}+ 2^{\frac{n}{3}-1}& \text{ if } n\equiv 0 \pmod 3,\\
3^{\frac{n-1}{3}}+ 2^{\frac{n-4}{3}} & \text{ if } n\equiv 1 \pmod 3,\\ 
4\cdot 3^{\frac{n-5}{3}}+3\cdot 2^{ \frac{n-8}{3}} & \text{ if } n\equiv 2 \pmod 3.\\ 
\end{cases}$$
For $n \ge 6$, the extremal graph is unique: if $n \equiv k \pmod 3$, where $k \in \{0,1,2\}$, the graph is obtained from a union of $k$ complete graphs $K_4$ and $\frac{n-4k}{3} = \floorfrac{n}{3}-k$ complete graphs $K_3$ by choosing one vertex from each of these complete graphs and connecting them to form a star of order $\floorfrac{n}{3}$. The centre of this star has to belong to a $K_4$, if there is (at least) one.
This is presented in Figure~\ref{fig:max_imax_connG} in the case $n \equiv 2 \pmod 3$.

\begin{figure}[h]
    \centering
    
    \begin{tikzpicture}
    {
    
    \draw[thick] (-5,4)--(1,5) -- (1,6);
    \draw[thick] (-2.5,4)--(1,5) -- (6,4);
    \draw[thick] (0,4)--(1,5);
    \draw[fill] (1,5) circle (0.1);
    
    \draw[thick] (1,5)--(0,7)--(2,7)-- cycle;
    \draw[fill] (0,7) circle (0.1);
    \draw[fill] (2,7) circle (0.1);

    \draw[thick, color=blue] (1,5) -- (1,6.2)--(2,7);
    \draw[thick, color=blue] (1,5) -- (1,6.2)--(0,7);
    \draw[fill=blue] (1,6.2) circle (0.1);
    
    \draw[thick, color=blue] (-5,4) -- (-5,2.8)--(-6,2);
    \draw[thick, color=blue] (-5,4) -- (-5,2.8)--(-4,2);
    \draw[fill=blue] (-5,2.8) circle (0.1);

    \foreach \x in {-5,-2.5,0,6}
    {
    \draw[fill] (\x,4) circle (0.1);
    \draw[fill] (\x-1,2) circle (0.1);
    \draw[fill] (\x+1,2) circle (0.1);
 
    \draw[thick] (\x,4) -- (1,5);

    \draw[thick] (\x+1,2)-- (\x,4) -- (\x-1,2)--cycle;
    }
    \node at (3,3) {\Large $\cdots$};
    
    }
    \end{tikzpicture}\\
    \caption{Connected graph with maximum value of $\imax$.}
    \label{fig:max_imax_connG}
\end{figure}
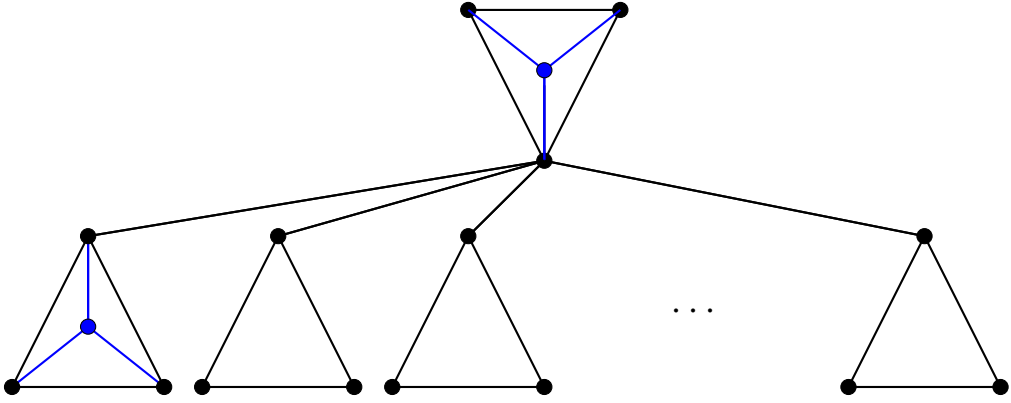

\subsection{The minimum number of maximal independent sets}

Whenever a graph contains at least one edge, we must clearly have $\imax(G) \ge 2$ (construct maximal independent sets greedily starting from each of the ends of one edge), and this bound is in fact attained by a star. As such, one can conclude that for trees or arbitrary connected graphs of order $n$, the best lower bound one can aim for is $\imax(G) \ge 2$.

%Here, we want to prove that the number of maximal independent sets is essentially exponential in the order.

If two leaves $v,v'$ of a tree have the same neighbour (we call such leaves \emph{twins}), then $v$ belongs to a maximal independent set if and only if $v'$ does, so we can consider them as essentially being one vertex. Generally, if two vertices $v$ and $v'$ of a graph have the same open neighbourhood, i.e., $N(v)=N(v')$, then every maximal independent set either contains both or neither.

So by duplicating vertices (adding new vertices with the same open neighbourhood as existing vertices), we can construct infinitely many graphs with the same number of maximal independent sets. This is also the reason why the question for the minimum number of maximal independent sets appears to be less interesting than its counterpart for the maximum. As we will see, it becomes more interesting if we forbid duplicated vertices: 

\begin{defi}
    A graph is called \emph{twin-free} if it has no two vertices with the same neighbours.
\end{defi}

Twin-free graphs have been studied in the past, e.g.~in~\cite{CHHL07} in connection with identifying codes.
Note that every graph can be reduced to a twin-free graph (its ``twin-free core'') by removal of duplicates without affecting the number of maximal independent sets. We therefore study the problem of determining the minimum of $\imax(G)$ for a twin-free graph $G$ of given order $n$ in different graph classes: specifically, arbitrary graphs, bipartite graphs, and trees.

We will see that for these three graph classes, the minimum of $\imax(G)$ is logarithmic, linear and exponential in terms of the order $n$, respectively.

Our first main result determines the minimum of $\imax(G)$ for arbitrary connected twin-free graphs of given order (in fact, it suffices to assume that there are no isolated vertices), and characterizes the extremal graph uniquely. We first prove that the extremal graph contains a clique of order $\imax(G)$, and that the neighbourhoods satisfy certain conditions. Using a result from extremal set theory, we can then conclude with the following theorem. This is presented in Section~\ref{sec:graphs}.

\begin{thr}\label{thr:imaxG}
    Let $k \ge 2$ be an integer.
    If $G$ is a twin-free graph of order $n$ without isolated vertices and $\imax(G) = k$, then $n \le 2^{k-1}+k-2$.
    Furthermore, equality holds only if the graph $G$ is formed by taking a clique $K_{k-1}$ and adding, for every non-empty vertex subset $S$ of this clique, a vertex whose neighbourhood is precisely $S$.
\end{thr}

The precise result for bipartite graphs is stated in the next theorem.
In the proof in Section~\ref{sec:bipgraphs}, we associate a unique maximal independent set to every vertex $v$ in one partition class, which together with the other partition class gives the lower bound on the number of maximal independent sets.

\begin{thr}\label{thm:conn_bipartite}
     Let $G$ be a twin-free bipartite graph of order $n \geq 2$ without isolated vertices.
    Then $\imax(G) \ge \ceilfrac{n}{2}+1$, and this inequality is sharp.
\end{thr}

The last result determines the minimum number of maximal independent sets in twin-free trees. This was previously proven by Taletski\u{\i} and Malyshev in~\cite{TM20}. We give a shorter proof in Section~\ref{sec:proofmainthr_trees}, which starts by deriving the bounds, in contrast to~\cite{TM20} where the extremal graphs are characterized first by excluding certain structures in the extremal graphs.
We use an inductive approach, based on a lemma of Wilf~\cite{Wilf86}.% and estimates on a general version of the dot product. 

\begin{thr}\label{thr:main}
Let $n \ge 4$ be an integer.
Then for every twin-free tree $T$ with $n$ vertices, we have
$$\imax(T) \ge 
\begin{cases}
4\cdot 3^{\frac{n}{5}-1} & \text{ if } n\equiv 0 \pmod 5,\\
5\cdot 3^{\frac{n-6}{5}} & \text{ if } n\equiv 1 \pmod 5,\\ 
2\cdot 3^{\frac{n-2}{5}} & \text{ if } n\equiv 2 \pmod 5,\\ 
8\cdot 3^{\frac{n-8}{5}} & \text{ if } n\equiv 3 \pmod 5,\\
3^{\frac{n+1}{5}}   & \text{ if } n\equiv 4 \pmod 5,
\end{cases}$$
and this inequality is sharp.
\end{thr}

For every $n$, Figure~\ref{fig:construction_mod5} depicts one example of a tree for which equality holds in Theorem~\ref{thr:main}. Here, the blue and red edges possibly have to be added depending on $n \pmod 5$.
For large $n$, there are multiple extremal trees. In Subsection~\ref{subsec:extrtrees} we briefly make some comments on the characterization, which also already appears in~\cite{TM20}. 

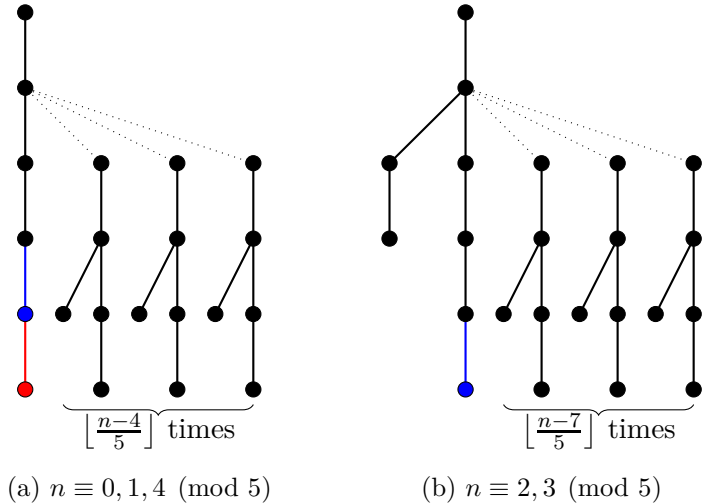
\begin{figure}[h]
\begin{center}
\begin{minipage}[b]{.3\linewidth}
\begin{center}
    \begin{tikzpicture}
    {

    \draw[thick] (1,5) -- (1,6);
    \draw[fill] (1,5) circle (0.1);
    \draw[fill] (1,6) circle (0.1);

    \foreach \x in {1}
    {
    \draw[thick, color=red] (1,\x) -- (1,\x+1);
    \draw[fill=red] (1,\x) circle (0.1);
    }
    
    \foreach \x in {2}
    {
    \draw[thick, color=blue] (1,\x) -- (1,\x+1);
    \draw[fill=blue] (1,\x) circle (0.1);
    }

    \foreach \x in {3,4}
    {
    \draw[thick] (1,\x) -- (1,\x+1);
    \draw[fill] (1,\x) circle (0.1);
    }

    \foreach \a in {2,3,4}
    {
    \draw[dotted] (\a,4) -- (1,5);
    \draw[fill] (\a,4) circle (0.1);
        \foreach \x in {1,2,3}
    	{
    	\draw[thick] (\a,\x) -- (\a,\x+1);
    	\draw[fill] (\a,\x) circle (0.1);
    	}
    \draw[thick] (\a,3) -- (\a-0.5,2);
    \draw[fill] (\a-0.5,2) circle (0.1);
    }
    \draw [decorate,decoration={brace,amplitude=4pt},xshift=0pt,yshift=0pt] (4,0.8) -- (1.5,0.8) node [black,midway,yshift=-0.35cm]{$\floorfrac{n-4}5 $ times};
	}
	\end{tikzpicture}\\
\subcaption{$n \equiv 0,1,4 \pmod 5$}
\label{fig:n1mod5}
\end{center}
\end{minipage}
\quad
\begin{minipage}[b]{.3\linewidth}
\begin{center}
    \begin{tikzpicture}
    {
    
    \draw[thick, color=blue] (1,1)--(1,2);
    \draw[fill=blue] (1,1) circle (0.1);

    \draw[thick] (0,4)--(1,5) -- (1,6);
    \draw[fill] (1,5) circle (0.1);
    \draw[fill] (1,6) circle (0.1);
    \draw[fill] (0,4) circle (0.1);

    \foreach \x in {3}
    {
    \draw[thick] (0,\x) -- (0,\x+1);
    \draw[fill] (0,\x) circle (0.1);
    }
    
    \foreach \x in {2,3,4}
    {
    \draw[thick] (1,\x) -- (1,\x+1);
    \draw[fill] (1,\x) circle (0.1);
    }
    
    \foreach \a in {2,3,4}
    {
    \draw[dotted] (\a,4) -- (1,5);
    \draw[fill] (\a,4) circle (0.1);
        \foreach \x in {1,2,3}
    	{
    	\draw[thick] (\a,\x) -- (\a,\x+1);
    	\draw[fill] (\a,\x) circle (0.1);
    	}
    \draw[thick] (\a,3) -- (\a-0.5,2);
    \draw[fill] (\a-0.5,2) circle (0.1);
    }
    \draw [decorate,decoration={brace,amplitude=4pt},xshift=0pt,yshift=0pt] (4,0.8) -- (1.5,0.8) node [black,midway,yshift=-0.35cm]{$\floorfrac{n-7}5$ times};
	}
	\end{tikzpicture}\\
\subcaption{$n \equiv 2,3 \pmod 5$}
\label{fig:n3mod5}
\end{center}
\end{minipage}
\end{center}
\caption{Constructions that attain equality in Theorem~\ref{thr:main}.}\label{fig:construction_mod5}
\end{figure}

Finally, in Section~\ref{sec:conc}, we conclude with some further directions related to the minimum of $\imax(G)$ for twin-free graphs in other classes and some other related questions.

\section{Twin-free graphs%with few maximal independent sets
}\label{sec:graphs}

We start with the proof that the minimum of the number of maximal independent sets in a connected twin-free graph is logarithmic in the order.
We first provide the extremal construction and give a short argument for the lower bound $\imax(G) > \log_2(n)$. The more precise Theorem~\ref{thr:imaxG} is proven afterwards in Subsection~\ref{subsec:actualproof}.

\begin{prop}\label{prop:imaxG_constr}
    There exists a connected twin-free graph with $n=2^{k-1}+k-2$ vertices for which $\imax(G)=k$.
\end{prop}

\begin{proof}
    Let $K_{k-1}$ be a clique on $k-1$ vertices.
    For every non-empty vertex subset $S$ of this clique,
    we add a new vertex whose neighbours are exactly the vertices in $S$.
    Then the graph $G$ obtained by this process has $k-1+(2^{k-1}-1)$ vertices, and all neighbourhoods are different, i.e., $G$ is twin-free.
    Every independent set contains at most one vertex of the clique $K_{k-1}$.
    We note that a maximal independent set $I$ is uniquely determined by $I \cap V(K_{k-1})$, since $V(G)\setminus V(K_{k-1})$ is an independent set.
    If $I \cap V(K_{k-1}) = \emptyset$, then $V(G)\setminus V(K_{k-1})$ becomes a maximal independent set on its own.
     If $I \cap V(K_{k-1}) =\{v\}$, then $I=V(G) \backslash N(v).$
    As such, we conclude that $\imax(G)=k$.
\end{proof}

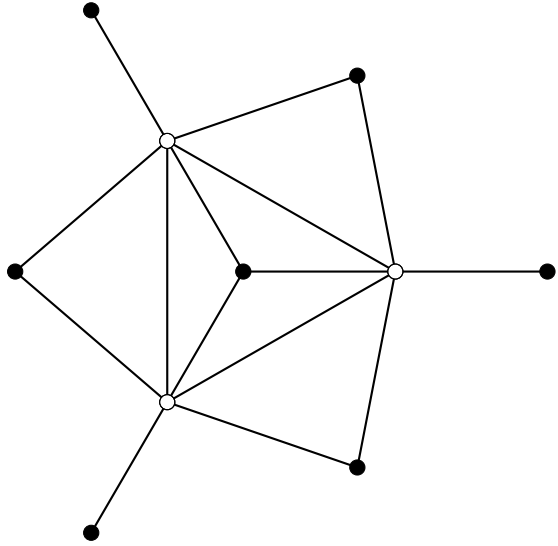
\begin{figure}[h]

\begin{center}
    \begin{tikzpicture}
    
    \draw[fill] (0,0) circle (0.1);
    \foreach \x in {1,2,3}
    {
    \draw[thick] (0:0)--(120*\x:2) -- (120*\x+120:2)--(120*\x+60:3)--(120*\x:2);
    
    \draw[fill] (120*\x+60:3) circle (0.1);
    \draw[thick] (120*\x:2)--(120*\x:4);
    \draw[fill] (120*\x:4) circle (0.1);
    }
    \foreach \x in {1,2,3}
    {
    \draw[fill=white] (120*\x:2) circle (0.1);
    \draw[] (120*\x:2) circle (0.1);
    }		
    \end{tikzpicture}
\end{center}
\caption{The construction in Proposition~\ref{prop:imaxG_constr} for $k=4$. The clique $K_3$ is indicated in white.}\label{fig:constructionGraph}
\end{figure}

On the other hand, we can prove a lower bound that is based on the following observation.

\begin{lem}\label{lem:twinfree<->uniquevectors}
    Two vertices $u$ and $v$ belong to the same maximal independent sets of a graph $G$ if and only if they are twins.
\end{lem}
\begin{proof}
    If $u$ and $v$ belong to the same maximal independent sets, then they cannot be neighbours, and they must have exactly the same neighbours: if for instance there is a vertex $w$ adjacent to $u$, but not to $v$, then there would be a maximal independent set containing $\{v,w\}$, but not~$u$. Hence $u$ and $v$ have to be twins. The converse is immediate.
\end{proof}

The following proposition already shows that the minimum of $\imax$ is of logarithmic order.

\begin{prop}\label{prop:imax<logn}
    For every connected twin-free graph $G$ with $n$ vertices, $\imax(G)>\log_2(n).$
\end{prop}

\begin{proof}
    Let $k=\imax(G)$, and let the maximal independent sets of $G$ be $I_1,\ldots,I_k$. For every vertex $v$, consider the associated vector $\Vec{v}=\left( 1_{v \in I_1}, 1_{v \in I_2}, \ldots, 1_{v \in I_k} \right)$. Since every vertex belongs to at least one maximal independent set, no vertex gets assigned the zero vector. Moreover, Lemma~\ref{lem:twinfree<->uniquevectors} shows that no two vertices have the same associated vector, as $G$ is assumed to be twin-free.
    Since there are only $2^k-1$ distinct non-zero vectors, we conclude that $n \le 2^k-1$, i.e., $k>\log_2(n)$.
\end{proof}

This argument is further refined in the following subsection.

\subsection{Proof of Theorem~\ref{thr:imaxG}}\label{subsec:actualproof}

Let us now prove %the following statement that is equivalent to
Theorem~\ref{thr:imaxG}, whose statement we first recall.% in a slightly stronger form (imposing the condition that $G$ has no isolated vertices instead of $G$ being connected).
 
\begin{customthm} {\bf \ref{thr:imaxG}}
    Let $k \ge 2$ be an integer.
    If $G$ is a twin-free graph of order $n$ without isolated vertices and $\imax(G) = k$, then $n \le 2^{k-1}+k-2$.
    Furthermore, equality holds only if the graph $G$ is formed by taking a clique $K_{k-1}$ and adding, for every non-empty vertex subset $S$ of this clique, a vertex whose neighbourhood is precisely $S$.
\end{customthm}

%It is immediate that $\imax(G) \ge 2$ if $n \ge 2$.
%As such, it is sufficient to consider a fixed value $k\ge 2.$
For $k=2$, the statement holds by Proposition~\ref{prop:imax<logn} and the fact that $i_{max}(K_3) = 3 >2.$

Suppose for contradiction that Theorem~\ref{thr:imaxG} does not hold, and let $k \ge 3$ be the minimum value for which this is the case.
Let $G$ be a graph of maximum order for which $\imax(G)=k$, and let $I_1,\ldots,I_k$ be its maximal independent sets. As in the proof of Proposition~\ref{prop:imax<logn}, we can conclude that every vertex of $G$ belongs to at least one of these sets.

 \begin{lem}\label{lem:def_U}
    Every maximal independent set $I_i$ has a vertex $u_i$ that belongs to no maximal independent set $I_j$ with $j \neq i$.
 \end{lem}
 \begin{proof}
     If this is not the case, add a vertex $u_i$ to the graph that is adjacent to all vertices except those in $I_i$. Then $I_i$ is extended by this new vertex, but all other maximal independent sets remain the same. Note that the new vertex does not belong to any maximal independent sets other than the extension of $I_i$. By Lemma~\ref{lem:twinfree<->uniquevectors}, the new graph $G'$ is still twin-free, and $\imax(G') = \imax(G)$, contradicting the choice of $G$. 
 \end{proof}
 
 \begin{lem}\label{lem:imaxGEchiGEomega}
     In every graph $H$, we have $\imax(H) \ge \chi(H) \ge \omega(H)$.
 \end{lem}
 \begin{proof}
     It is sufficient to observe that every colour class of an optimal proper colouring is an independent set, but the union of two colour classes is not.
     Since every independent set can be (e.g.~greedily) extended to at least one maximal independent set, $\imax(H) \ge \chi(H)$ is immediate.
     The inequality $\chi(H) \ge \omega(H)$ is well known.
 \end{proof}

 In the graph $G$ that we took as counterexample to Theorem~\ref{thr:imaxG}, any two vertices $u_i$ and $u_j$ as defined in Lemma~\ref{lem:def_U} with $i \neq j$ have to be adjacent. Otherwise, $\{u_i,u_j\}$ could be extended to a maximal independent set, contradicting the choice of $u_i$ and $u_j$. So $U=\{u_1, \ldots, u_k\}$ spans a complete graph, implying that $\omega(G) \geq k = \imax(G)$, which means that we have equality in Lemma~\ref{lem:imaxGEchiGEomega} when applied to $G$.

\begin{lem}
A vertex $v$ belongs to $I_j$ if and only if it is not adjacent to $u_j$.
\end{lem}

\begin{proof} If $v$ belongs to $I_j$, then it can clearly not be adjacent to $u_j$ (which also belongs to $I_j$). Conversely, if $v,u_j$ are not adjacent, then $\{v,u_j\}$ can be extended to a maximal independent set, which must necessarily be $I_j$ (by Lemma~\ref{lem:def_U}).
\end{proof}

Remembering that the vectors $\Vec{v}$ as defined in the proof of Proposition~\ref{prop:imax<logn} are all distinct and neither equal to the all-$0$ vector (as every vertex belongs to a maximal independent set) nor the all-$1$ vector (the corresponding vertex would be isolated), we know that every vertex has at least one neighbour in the clique induced by $U$, and no two vertices $v,v' \notin U$ satisfy $N(v) \cap U = N(v') \cap U$.

\begin{lem}\label{lem:adj_cond}
    Two vertices $v,v'$ are adjacent if and only if $U \subseteq N(v) \cup N(v')$.
\end{lem}
\begin{proof}
    If $v,v'$ are adjacent and $u \in U$ does not belong to $N(v) \cup N(v')$, then $\{u,v\}$ and $\{u,v'\}$ can be extended to distinct maximal independent sets containing $u$. This contradicts the construction of $U$ (Lemma~\ref{lem:def_U}), so $U \subseteq N(v) \cup N(v')$ in this case.
    
    Conversely, if $v,v'$ are not adjacent, then the set $\{v,v'\}$ can be extended to a maximal independent set, which must contain a vertex of $U$ by construction. As this vertex is not in $N(v) \cup N(v')$, $U$ is not fully contained in $N(v) \cup N(v')$ in this case.
\end{proof}

If the union of the neighbourhoods of the vertices in an independent set in $V(G) \setminus U$ contains all of $U$, it can be extended to a maximal independent set without any vertices in $U$, leading to a contradiction again.
Equivalently, if the union of the neighbourhoods of some vertices in $V(G) \setminus U$ contains all of $U$, then it cannot be an independent set, so it must contain two adjacent vertices whose neighbourhoods cover $U$ by Lemma~\ref{lem:adj_cond}. This observation naturally leads us to the following concept.

\begin{defi}\label{defi:unionefficient}
    We call a family $\F \subseteq 2^{[n]}$ \emph{union-efficient} if for every subfamily $\{A_1, \ldots, A_{m}\}$ of $\F$ for which $\cup_{i \in [m]} A_i = [n]$, there are two indices $i,j \in [m]$ for which $A_i \cup A_j =[n].$
\end{defi}

For every vertex $v_i \in V(G) \setminus U$, we consider the index set $A_i=\{j \mid u_j \in N(v_i), 1 \le j \le k\}$. As explained before, the family $\F \subseteq 2^{[k]}$ consisting of all these $A_i$ is union-efficient.

Note that $\F$ does not include $\emptyset, [k]$ (as observed before), nor $[k]\setminus j$ for any $j \in [k]$. 
The latter holds because $A_i = [k] \setminus j$ would mean that
$v_i$ is adjacent to (a) all vertices in $U$ except $u_j$, and
(b) precisely those vertices in $V(G) \setminus U$ that are adjacent to $u_j$ (by Lemma~\ref{lem:adj_cond}). But then $v_i$ and $u_j$ would be twins, a contradiction. It is easy to see that $\F$ stays union-efficient if we add these $k+2$ sets to $\F$.

Formulated in the terminology of union-efficient families (the equivalence is explained in~\cite{CS22+}), the following theorem is a result by Milner (see~\cite{Erdos71}), with the uniqueness result proven by Bollob\'as and Duchet \cite{BD83}, and by Mulder \cite{Mulder83}.

\begin{thr}\label{thr:mainthr_extrset}
    Let $n \ge 3$, and let $\E \subseteq 2^{[n]}$ be a union-efficient family.
    Then $\lvert \E \rvert \le 2^{n-1}+n.$
    Equality is attained if and only if (up to isomorphism) $\E=2^{[n-1]} \cup \binom{[n]}{\ge n-1}$ (i.e., $\E$ contains all subsets of $[n-1]$ and all subsets of $[n]$ with at least $n-1$ elements).
\end{thr}

Theorem~\ref{thr:mainthr_extrset} implies that $\lvert \F \rvert \le 2^{k-1}-2$, thus
$$|V(G)| = |U| + |V(G) \setminus U| \le k + 2^{k-1}-2,$$
and equality holds if and only if (up to renaming) $\F=2^{[k-1]} \setminus \{[k-1], \emptyset \}$. In this case, $G$ consists of the clique induced by $u_1,\ldots,u_{k-1}$ and the $2^{k-1}-1$ vertices in $V(G) \setminus \{u_1,\ldots,u_{k-1}\}$ whose neighbourhoods are precisely all the different nonempty subsets of $\{u_1,\ldots,u_{k-1}\}$ ($u_k$ is the unique vertex adjacent to all of them). This is precisely the characterization of the extremal graph described in Theorem~\ref{thr:imaxG}, completing our proof. \qed

Note that a twin-free graph has at most one isolated vertex, and that adding an isolated vertex does not change the number of maximal independent sets. Thus, if we drop the condition that there are no isolated vertices, we obtain the following version of Theorem~\ref{thr:imaxG}.

% Note that connectedness was only used in one place in our entire proof: to rule out isolated vertices. If we drop this condition, we obtain the following version of Theorem~\ref{thr:imaxG}.

\begin{cor}\label{cor:disconn}
    Let $k \ge 2$ be an integer.
     If $G$ is a twin-free graph of order $n$ and $\imax(G) = k$, then $n \le 2^{k-1}+k-1$.
     Furthermore, equality holds only if the graph $G$ is formed by taking a clique $K_{k-1}$ and adding, for every vertex subset $S$ of this clique, a vertex whose neighbourhood is precisely $S$.
\end{cor}

\section{Twin-free bipartite graphs}\label{sec:bipgraphs}

In this section, we show that the minimum value of $\imax(G)$ for twin-free bipartite graphs is linear in the order. We start by proving Theorem~\ref{thm:conn_bipartite}, which we recall for convenience.
\begin{customthm} {\bf \ref{thm:conn_bipartite}}
     Let $G$ be a twin-free bipartite graph of order $n \geq 2$ without isolated vertices.
    Then $\imax(G) \ge \ceilfrac{n}{2}+1$, and this inequality is sharp.
\end{customthm}

\begin{proof}
    For the lower bound, we prove a stronger statement:
    if $G$ is a twin-free bipartite graph without isolated vertices whose bipartition classes $A, B$ have sizes $a \le b$ ($b \le 2^a -1$ is necessary for existence), then $\imax(G) \ge b+1.$
    
    For every vertex $v \in B$, consider the set $I_v = (A \setminus N(v) ) \cup \{ u \in B \mid N(u) \subseteq N(v)\}$.
    This is a maximal independent set by construction.
    Since for every two vertices $u,v \in B$ their neighbourhoods $N(u)$ and $N(v)$ are different, we have $I_u \neq I_v$ if $u \neq v.$
    Finally, $A$ is also a maximal independent set (since $G$ has no isolated vertices) that does not coincide with any of the $I_v$. Thus $G$ has at least $b+1$ maximal independent sets.
    Since $b \ge \ceilfrac n2,$ the lower bound is clear.

    Equality is attained for example by taking a balanced complete bipartite graph $K_{\floorfrac n2, \ceilfrac n2}$ and removing a matching $M$ of size $\ceilfrac n2 - 1$.
    The graph $K_{\floorfrac n2, \ceilfrac n2} \setminus M$ is clearly bipartite, connected and twin-free (note that a bipartite graph is twin-free if and only if its bipartite complement is, provided that neither of the two has two isolated vertices).
    It has two types of maximal independent sets: each of the bipartition classes is a maximal independent set, and for each of the $\ceilfrac n2 - 1$ edges in $M$, the ends form a maximal independent set of cardinality $2$. There are no others: once two vertices from the same bipartition class are contained in an independent set $I$, no vertices from the other class can be included. By maximality, $I$ must be one of the bipartition classes in this case. If an independent set $I$ contains vertices from both classes, then they have to be the ends of an edge in $M$, as they would otherwise be adjacent. Hence $\imax\left(K_{\floorfrac n2, \ceilfrac n2} \setminus M\right) = \ceilfrac n2 + 1$.
\end{proof}

The same bound seems to hold more generally for triangle-free graphs, but we do not have a proof at this point.

\begin{conj}\label{conj:K3free}
    Let $G$ be a twin-free and triangle-free graph of order $n$ without isolated vertices.
    Then $\imax(G) \ge \ceilfrac{n}{2}+1.$
    Furthermore, if $n$ is even, graphs that attain equality are bipartite.
\end{conj}

If we drop the condition that there are no isolated vertices, the minimum only changes by at most $1$ by the same reasoning that gave us Corollary~\ref{cor:disconn}: a twin-free graph has at most one isolated vertex, and adding an isolated vertex does not change the number of maximal independent sets. The same is true for triangle-free graphs under Conjecture~\ref{conj:K3free}.

\begin{cor}
For every twin-free bipartite graph $G$ of order $n$, we have $\imax(G) \ge \ceilfrac{n-1}{2}+1=\floorfrac{n}{2}+1$, and this inequality is sharp.
\end{cor}

The graphs in the proof that were constructed to show that the inequality is sharp are not unique---there are many more extremal graphs, see Table~\ref{tab:nrExtBipK3freeG}.

\begin{table}[h]
    \centering
    \begin{tabular}{|c|c|c|c|}
	\hline
	$n$ & $\min \imax(G)$ & \# extremal bip. graphs & \# extremal $K_3$-free graphs \\
	\hline
	4&3&1&1	 \\ 
	5&4&1&1	 \\ 
    6 &4& $2$&2 \\ 	
    7&5 & $4$&5\\
    8&5&4&4\\
    9&6& 16&18\\
    10&6&11& 11\\
    11&7& 73&79\\
    12&7& 33&33\\
	\hline
	\end{tabular}
    \caption{The number of connected twin-free bipartite/triangle-free graphs for which the minimum in Theorem~\ref{thm:conn_bipartite} is attained.}
    \label{tab:nrExtBipK3freeG}
\end{table}  
Let us also present a bijection with certain binary matrices.

\begin{prop}\label{prop:equivmatr}
    The connected twin-free bipartite graphs of order $2k$ for which $\imax(G)=k+1$ are in one-to-one correspondence with $k\times k$-binary matrices satisfying the following conditions:
    \begin{itemize}
        \item there is an all-$1$ row and an all-$1$ column,
        \item all columns are distinct, and all rows are distinct,
        \item the union (bitwise maximum) of any two rows is a row of the matrix itself.
    \end{itemize}
\end{prop}

\begin{proof}
In the proof of Theorem~\ref{thm:conn_bipartite}, it was shown that $\imax(G) \ge \max\{ \abs{A}, \abs{B}\}+1$, where $V(G)=A \cup B$ is the bipartition of $G.$
Hence the equality $\imax(G)=k+1$ can only occur among balanced bipartite graphs, i.e., when $\abs{A}=\abs{B}$. The bipartite graph can now be presented by its reduced adjacency matrix $\M$, which is a $k \times k$-matrix.
From the proof, we also conclude that there is a vertex $v \in B$ for which $I_v=B$ and thus $N(v)=A.$
Similarly, there is a vertex $w \in A$ with $N(w)=B$. Thus $\M$ has a row and a column containing only $1$s.
The graph $G$ being twin-free is equivalent to $\M$ having distinct rows and distinct columns.

If there is a bipartition class, say $B$, of $G$ for which the union of the neighbourhoods of some of its vertices, $B_1 \subseteq B$ ($B_1 \not= \emptyset$), is not equal to the neighbourhood of a vertex in $B$, then $A \setminus N(B_1) \cup \{ u \in B \mid N(u) \subseteq N(B_1) \}$ would be another maximal independent set not counted in the proof of Theorem~\ref{thm:conn_bipartite}.
So for every $B_1 \subseteq B$, there is a $b \in B$ with $N(B_1)=N(b).$
The latter is the case if and only if this is true for every subset $B_1$ of size $2,$ which is equivalent with the union of any $2$ rows being a row itself.

In the reverse direction, assuming $\M$ does satisfy the conditions, we know that it is the adjacency matrix of a connected twin-free bipartite graph $G$.
Assume a maximal independent set of $G$ different from $A$ consists of precisely the vertices in $A_1 \subseteq A$ and $B_1 \subseteq B$. Due to the third condition, there is a vertex $b$ for which $N(B_1)=N(b).$
Then $B_1$ has to be precisely equal to $\{u \in B \mid N(u) \subseteq N(b)\}$, and $A_1=A \backslash N(b)$ since $A_1 \cup B_1$ is a maximal independent set.
Therefore the maximal independent sets are exactly those of the form $\{I_b \mid b \in B\} \cup \{A\}$, and we have $\imax(G)=k+1$.
We conclude that we really have a bijection.
\end{proof}

\begin{remark}
    This also gives a combinatorial proof that among square binary matrices $\M$ with all columns resp. rows distinct and containing an all-$1$ row and an all-$1$ column,
    the condition that the union of any set of columns of $\M$ is a column of $\M$ is equivalent with the condition that the union of any set of rows of $\M$ is a row of $\M$. 
\end{remark}

From this characterization, it can be derived that there are at least exponentially many bipartite graphs $G$ of order $2k$ for which $\imax(G)=k+1$.
For example, $2^{\floorfrac k3}$ different matrices satisfying the constraints in Proposition~\ref{prop:equivmatr} can be obtained by taking a $k \times k$ binary lower-triangular matrix $\M$ all of whose lower-triangular elements are $1$ and then additionally setting some entries $\M_{3i+1,3i+3}$ equal to $1$ for some $0 \le i < \floorfrac k3.$
For $k=6$, this gives four possibile matrices of this form (a red entry can be either $0$ or $1$)
$$ \M= \begin{pmatrix}
1 & 0& \textcolor{red}{0/1} &0&0&0 \\
1&1&0&0&0&0 \\
1&1&1&0&0&0 \\
1&1&1&1&0&\textcolor{red}{0/1} \\
1&1&1&1&1&0 \\
1&1&1&1&1&1 \\
\end{pmatrix}.$$

\section{Twin-free trees}
\label{sec:proofmainthr_trees}

In this section, we consider the problem for trees. Let us start again by recalling the theorem we want to prove. We first define a function $f$ by $f(1)=1, f(2)=f(3)=2$, and for $n \ge 4$
    $$f(n)=\begin{cases}
4\cdot 3^{\frac{n}{5}-1} & \text{ if } n\equiv 0 \pmod 5,\\
5\cdot 3^{\frac{n-6}{5}} & \text{ if } n\equiv 1 \pmod 5,\\ 
2\cdot 3^{\frac{n-2}{5}} & \text{ if } n\equiv 2 \pmod 5,\\ 
8\cdot 3^{\frac{n-8}{5}} & \text{ if } n\equiv 3 \pmod 5,\\
3^{\frac{n+1}{5}}   & \text{ if } n\equiv 4 \pmod 5.
\end{cases}
$$
The main theorem of this section can thus be stated as follows.

\begin{customthm} {\bf \ref{thr:main}}
Let $n \ge 4$ be an integer.
Then for every twin-free tree $T$ with $n$ vertices, we have $\imax(T) \ge f(n)$, and this inequality is sharp.
\end{customthm}

Note that the inequality $\imax(T) \ge f(n)$ is also true for $n \in \{1,2,3\}$, although the statement is void for $n=3$ as there are no twin-free trees with $3$ vertices (however, $\imax(P_3)=2$).

For $ n \le 8$, the values can be verified by determining $\imax(T)$ for all trees of order $n$ (these are the base cases of our induction proof).
So from now on, we consider $n \ge 9$ and assume that the statement has been proven for every smaller order.

In the proof of Theorem~\ref{thr:main}, we will use the following estimates.

\begin{lem}\label{lem:estimates_f}
        The following statements hold:
        \begin{enumerate}[label=(\alph*)]
            \item\label{itm:1}
             except for the pair $(n,m)=(3,3)$, we have $f(n) \cdot f(m) \ge f(n+m)$ for all $n,m \ge 2$, and
            \item\label{itm:2} $f(n-1) \cdot f(m-1) \ge f(n+m-1)$ for all $n,m\ge 5.$
        \end{enumerate}
    \end{lem}
    \begin{proof}
        Note that for $n \ge 5$, the function $g(n)=f(n-1)3^{-\frac{n}{5}}\ge 1$ only depends on the residue class modulo $5$, namely
        $$g(n)=\begin{cases}
        1   & \text{ if } n\equiv 0 \pmod 5,\\
4\cdot 3^{-\frac{6}{5}} & \text{ if } n\equiv 1 \pmod 5,\\
5\cdot 3^{-\frac{7}{5}} & \text{ if } n\equiv 2 \pmod 5,\\ 
2\cdot 3^{-\frac{3}{5}} & \text{ if } n\equiv 3 \pmod 5,\\ 
8\cdot 3^{-\frac{9}{5}} & \text{ if } n\equiv 4 \pmod 5.\\
\end{cases}$$
A direct verification shows that
$$\frac{g(i)g(j)}{g(i+j)} \in \Big\{1,\frac{25}{24},\frac{16}{15}, \frac{10}{9}, \frac{32}{27}\Big\},$$
thus in particular $g(i)g(j) \ge g(i+j)$ in all cases. So it follows that
$$f(n-1)f(m-1) = 3^{(n+m)/5} g(n)g(m) \ge 3^{(n+m)/5} g(n+m) = f(n+m-1)$$
for all $n,m \ge 5$. Thus statement (b) holds.
Since $f(n)$ is increasing for $n \ge 4$, we have
$$f(n)f(m) \ge f(n+m+1) > f(n+m)$$
for all $n,m \ge 4$. Thus we only need to verify (a) in those cases where either $n$ or $m$ (without loss of generality $n$) is equal to $2$ or $3$. If both $n$ and $m$ are equal to $2$ or $3$, then the inequality holds unless $n = m = 3$, as $f(6) = 5 > f(2)^2 = f(2)f(3) = f(3)^2 = 4 = f(5) > f(4) = 3$. 

Moreover, since $f(2) = f(3) = 2$, it actually suffices to consider $n = 3$ (again because $f$ is an increasing function: $f(3)f(m) \ge f(m+3)$ implies $f(2)f(m) = f(3)f(m) \ge f(m+2)$). The inequality $f(3)f(m) = 2f(m) \ge f(m+3)$ is easily verified for $m \ge 4$ by checking all five possible cases modulo $5$ and finding that
$$\frac{f(3)f(m)}{f(3+m)} \in \Big\{1,\frac{16}{15}, \frac{10}{9}\Big\}.$$
\end{proof}

Next, we explain the inductive idea used by Wilf~\cite{Wilf86}.
    Let $x$ be a leaf of $T$ and $y$ its neighbour.
    Let $u_1, u_2, \ldots, u_r$ be the neighbours of $y$ different from $x$.
    Let $U_i$ be the component (subtree) of $T \setminus \{x,y\}$ that contains $u_i$, and let $W_{i,j}$, $1 \le j \le s_i$, be the components of $U_i \setminus u_i$.
    Finally, let $w_{i,j}$ be the neighbour of $u_i$ belonging to $W_{i,j}.$
    This is illustrated in Figure~\ref{fig:tree_with_UW}.
    
\begin{figure}[h]
    \begin{center}
    \begin{tikzpicture}
    {
    
    \draw[thick] (-4,4)--(1,5) -- (1,6);
    \draw[fill] (1,5) circle (0.1);
    \draw[fill] (1,6) circle (0.1);
    \draw[fill] (-4,4) circle (0.1);

    \foreach \x in {-4,0.5,6}
    {
    \draw[fill] (\x-1,3) circle (0.1);
    \draw[fill] (\x+1,3) circle (0.1);
    \draw[] (\x-1,2) ellipse (0.75cm and 1cm);
    \draw[] (\x+1,2) ellipse (0.75cm and 1cm);
    \draw[thick] (\x,4) -- (1,5);
    \draw[fill] (\x,4) circle (0.1);
    \draw[dashed] (\x,2.22) ellipse (1.84cm and 2.35cm);
    \draw[thick] (\x+1,3)-- (\x,4) -- (\x-1,3);
    \node at (\x,2) {$\ldots$};
    }
    
    \node at (3.25,2.22) {\Large $\cdots$};
    
    \node at (1,6.3) {$x$};
    \node at (1.25,5.3) {$y$};
    \node at (-4,4.3) {$u_1$};
    \node at (0.3,4.3) {$u_2$};
    \node at (6,4.3) {$u_r$};
    \node at (-5.15,3.3) {$w_{1,1}$};
    \node at (-2.85,3.3) {$w_{1,s_1}$};
    \node at (-0.65,3.3) {$w_{2,1}$};
    \node at (1.65,3.3) {$w_{2,s_2}$};
    \node at (7.15,3.3) {$w_{r,s_r}$};
    \node at (4.85,3.3) {$w_{r,1}$};
    \node at (-5,2) {$W_{1,1}$};
    \node at (-3,2) {$W_{1,s_1}$};
    \node at (-0.5,2) {$W_{2,1}$};
    \node at (1.5,2) {$W_{2,s_2}$};
    \node at (7,2) {$W_{r,s_r}$};
    \node at (5,2) {$W_{r,1}$};
    \node at (-4,0.5) {$U_{1}$};
    \node at (0.5,0.5) {$U_{2}$};
    \node at (6,0.5) {$U_{r}$};
    }
	\end{tikzpicture}\\
\end{center}
\caption{Tree with subtrees.}\label{fig:tree_with_UW}
\end{figure}
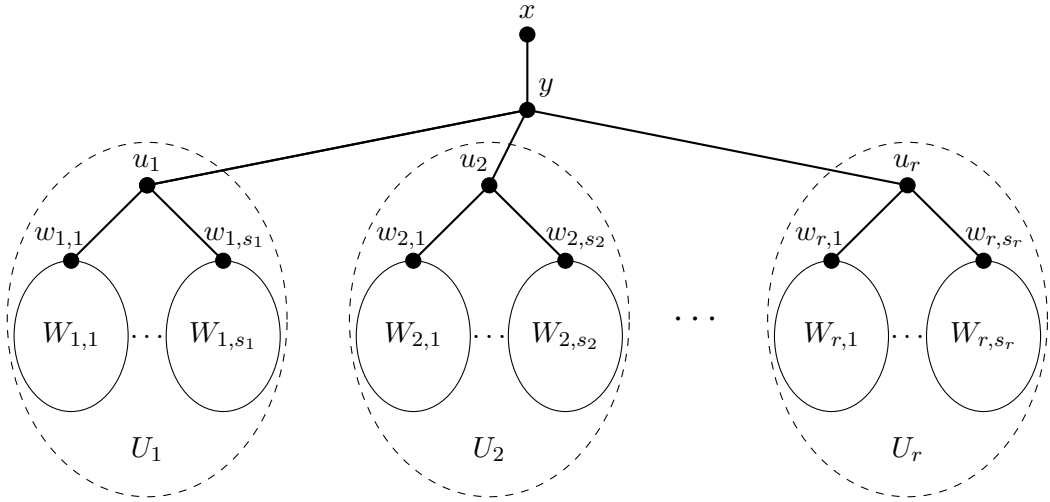

Then the following formula holds, see~\cite[Lemma~1]{Wilf86}.
    \begin{lem}[\hspace{1sp}\cite{Wilf86}]\label{lem:formula_Wilf}
        If $T$ is a tree decomposed as in Figure~\ref{fig:tree_with_UW}, then
        $$ \imax(T)= \prod_{i=1}^r \imax (U_i) + \prod_{i=1}^r \prod_{j=1}^{s_i} \imax(W_{i,j}).$$
    \end{lem}
With this formula, it is immediate to prove that the five constructions in Figure~\ref{fig:construction_mod5} give equality in Theorem~\ref{thr:main}. 

    Let us now continue with our induction proof. Assume that $T$ is a twin-free tree of order $n \ge 9$. In the decomposition of Figure~\ref{fig:tree_with_UW}, assume without loss of generality that $x$ is a leaf that is an end of a diametral path. This implies that all but one of $y$'s neighbours are leaves, and since we are only considering twin-free trees, it actually implies that $y$ has degree $2$, i.e., $r = 1$. Writing $s = s_1$, $U = U_1$, $W_{j} = W_{1,j}$, $u = u_1$ and $w_j = w_{1,j}$ for simplicity, the formula in Lemma~\ref{lem:formula_Wilf} is reduced to
 \begin{equation}\label{eq:wilf_reduced}
\imax(T)= \imax (U) + \prod_{j=1}^{s} \imax(W_{j}).
\end{equation}

Note that $\lvert U \rvert = n-2$. The core of our induction lies in the following lemma.

    \begin{lem}\label{lem:U&W_estimates}
      At least one of the following holds:
        \begin{enumerate}
\item\label{item:ineq_1} $\imax(U)\ge f(n-3)$ and $\displaystyle \prod_{j=1}^{s} \imax(W_{j}) \ge f(n-3)$, or
\item\label{item:ineq_2} $\imax(U)\ge f(n-2)$ and $\displaystyle \prod_{j=1}^{s} \imax(W_{j}) \ge f(n-4)$, or
\item\label{item:ineq_3} $\imax (U) \ge 2f(n-5)$ and $\displaystyle\prod_{j=1}^{s} \imax(W_{j}) \ge f(n-5)$.
        \end{enumerate}
    \end{lem}
    \begin{proof}
    We divide the proof into two cases, depending on the root degree of $U$.
    
        \textbf{Case 1: $s=1$.}\\
        Observe that if $U$ is not twin-free
        (in which case $u$ and a leaf of $U$ are at distance $2$ from each other), then $W_{1}$ is twin-free and by the induction hypothesis $\imax(U)= \imax(W_{1})\ge f(n-3)$, so the first of the three statements holds. 
        If $U$ is twin-free, $W_{1}$ might not be, but can in this case be made twin-free by removing a vertex without affecting $\imax$. Hence we have
        $\imax(U)\ge f(n-2)$ and $\imax(W_{1})\ge f(n-4)$, so the second statement holds.
        
        \textbf{Case 2: $s\ge 2$.}\\
        First we observe that $U$ is now always a twin-free tree and thus $\imax(U)\ge f(n-2)$ by the induction hypothesis.
        We also observe that $ \imax(W_j) \ge f( \lvert W_j \rvert )$ if $W_j$ is twin-free and otherwise $ \imax(W_j) \ge f( \lvert W_j \rvert -1)$ by the same argument as in Case 1; only the vertex $w_j$ can become a twin of another vertex, and in that case $W \setminus w_j$ is twin-free.

        If $\lvert W_j \rvert \le 4$, then we have $\imax(W_j) = f(\lvert W_j \rvert)$: in each case, there is only one possible tree up to isomorphism ($W_j$ cannot be a star of order $4$, since $T$ would then not be twin-free). Hence $ \imax(W_j) < f( \lvert W_j \rvert)$ can only happen if $\lvert W_j \rvert \ge 5$.

        We consider two subcases now: assume first that $\lvert W_j \rvert > 1$ for all $j$. We have $\imax(W_j) \ge f(\lvert W_j \rvert - 1)$ whenever $\lvert W_j \rvert \ge 5$, and iterating part (b) of Lemma~\ref{lem:estimates_f} gives us
        $$\prod_{j: \lvert W_j \rvert \ge 5} \imax(W_j) \ge f \left( \left(\sum_{j: \lvert W_j \rvert \ge 5} \lvert W_j \rvert \right) - 1 \right),$$
        provided there are $W_j$ of order $5$ or greater.
        For all $j$ with $W_j < 5$, we can use the fact that $\imax(W_j) = f(\lvert W_j \rvert)$. Applying part (a) of the same lemma repeatedly now, we end up with
        \begin{equation}\label{eq:prod_iteration}
            \prod_{j=1}^{s} \imax(W_j) \ge f \left( \left(  \sum_{j=1}^{s} \lvert W_j \rvert \right) - 1 \right) = f(n-4),
        \end{equation}
        so statement \ref{item:ineq_2} holds in this case. If there are no $W_j$ of order $5$ or greater, we can skip the first step and only apply part (a) of Lemma~\ref{lem:estimates_f}. The exceptional case $(3,3)$ only occurs at most once in the process, and we have $f(3)^2 = 4 = f(5)$, thus~\eqref{eq:prod_iteration} still applies. In either case, we are done.
        
        Let us finally consider the case that there is a $j$ (without loss of generality $j=1$) such that $\lvert W_j \rvert = 1$. There can only be one, as $T$ is assumed to be twin-free. The same argument as in the other subcase now yields
        $$\prod_{j=1}^{s} \imax(W_j) = \prod_{j=2}^{s} \imax(W_j) \ge f \left( \sum_{j=2}^{s} \lvert W_j \rvert- 1 \right) = f(n-5).$$
        Moreover, this bound can be improved if there is a twin-free $W_j$ of order at least 4: in this case, we have $\imax(W_j) \ge f(\lvert W_j \rvert)$. Iterating part (b) of Lemma~\ref{lem:estimates_f} then yields
        $$\prod_{j: \lvert W_j \rvert \ge 4} \imax(W_j) \ge f \left( \sum_{j: \lvert W_j \rvert \ge 4} \lvert W_j \rvert \right)$$
        and thus~\eqref{eq:prod_iteration} again, in which case we are done. Likewise, since $f(2) = f(3) = 2$, we can improve the bound to $f(n-4)$ if there is a $W_j$ of order $2$.

        So if~\eqref{eq:prod_iteration} does not hold, then $W_j$ cannot be twin-free for any $j \ge 2$ (there are no twin-free trees of order $3$, and all other orders have been ruled out). This is only possible if the degree of $w_{j}$ is $2$. Moreover, $\imax(W_j \setminus w_{j}) = \imax(W_j)$ in this case.
        Taking $w_{1}$ (which by our assumption is a leaf) as the new root, we can apply Lemma~\ref{lem:formula_Wilf} to $U$ and deduce that $$\imax(U) = \prod_{j=2}^{s} \imax(W_j) + \prod_{j=2}^{s} \imax(W_j\setminus w_{j})= 2 \prod_{j=2}^{s} \imax(W_j) \ge 2f(n-5).$$
        Thus statement~\ref{item:ineq_3} applies in this case, which completes the proof.
    \end{proof}

Applying the inequalities in Lemma~\ref{lem:U&W_estimates} to the formula in~\eqref{eq:wilf_reduced}, we obtain
$$\imax(T) \geq \min \big(2f(n-3),f(n-2)+f(n-4),3f(n-5) \big).$$
It is however straightforward to verify from the definition of $f(n)$ that $f(n) \le 2f(n-3)$ (with equality if and only if $n \equiv 0,2,3 \pmod 5$), $f(n) \le f(n-2) + f(n-4)$ (with equality if and only if $n \equiv 1,3 \pmod 5$) and $f(n) = 3f(n-5)$ for all $n \ge 9$, so the inequality $\imax(T) \geq f(n)$ follows immediately as a consequence, which finally completes the proof of Theorem~\ref{thr:main}. \qed

Theorem~\ref{thr:main} also easily extends to forests, as the following corollary shows.

\begin{cor}
For every twin-free forest $F$ of order $n \ge 2$, we have $\imax(F) \ge f(n-1)$.
\end{cor}

\begin{proof}
We use induction on $n$. For $n \le 4$, the inequality is easy to check, so let us consider $n \ge 5$. Suppose that $F$ is disconnected and has a connected component $T$ of order $m$ with $n-1>m \ge 2$. Taking the smallest such component, we can assume that $m \le \frac{n}{2}$ and thus $n - m \ge \ceilfrac{n}{2} \ge 3$. By Theorem~\ref{thr:main} and the induction hypothesis, we have
$$\imax(F) = \imax(T) \cdot \imax(F - T) \ge f(m) \cdot f(n-m-1) \ge f(n-1),$$
where the final inequality follows from part (a) of Lemma~\ref{lem:estimates_f}. Note here that $m \neq 3$, since there are no twin-free trees of order $3$. The only remaining possibilities are that $F$ is connected (i.e., a tree), or that $F$ consists of a tree and one isolated vertex. In either case, the conclusion follows from Theorem~\ref{thr:main}.
\end{proof}

\subsection{The extremal trees}\label{subsec:extrtrees}

The extremal trees can be constructed iteratively by tracking the cases of equality in the proof. A complete characterization is provided in~\cite{TM20}. Here we add a small correction for $n=8$: there are three extremal trees (rather than two as claimed in~\cite{TM20}), see Figure~\ref{fig:extrTn8}.

%The first one, see Figure~\ref{fig:extrTn8} at the left, can be considered as the tree in $\R(3,0,0,0).$ The other two as the ones in $\R(0,0,0,1).$
Due to the different constructions, it is not too surprising that the number of extremal trees is not monotone as a function of $n$. The number of extremal trees for $4 \le n \le 19$ is summarized in Table~\ref{tab:NrextrTrees}.

    \begin{table}[h]
        \centering
            \begin{tabular}{|c|c|}
	\hline
	$n$ & \# Extremal trees \\
	\hline
	$\{4,5,7,9\}$ & $1$	 \\ $\{6,14\}$ & $2$	 \\
	$\{8,10,11,12\}$ & $3$	 \\ 
	13& $11$	 \\ 
	$\{15,16\}$& $12$	 \\ 	17& $10$	 \\ 
    18& $60$\\
    19& $5$\\
	\hline
	\end{tabular}
        \caption{Number of extremal trees (i.e., trees satisfying $\imax(T)=f(n)$) of order $n$.}
        \label{tab:NrextrTrees}
    \end{table}  
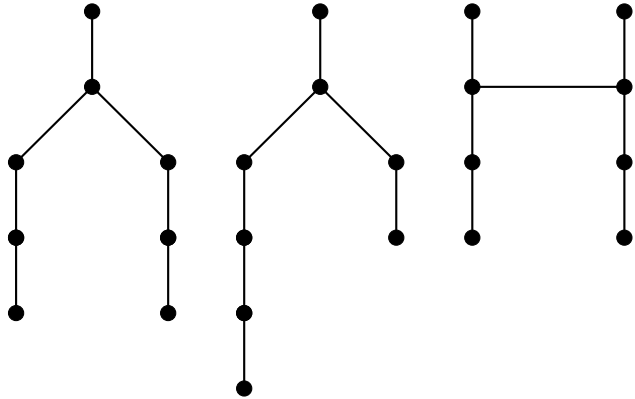
\begin{figure}[h]
\begin{center}
    \begin{tikzpicture}

    \draw[thick] (1,5) -- (1,6);
    \draw[fill] (1,5) circle (0.1);
    \draw[fill] (1,6) circle (0.1);
    \draw[thick] (1,5) -- (0,4);
    \draw[thick] (1,5) -- (2,4);
    
    \foreach \x in {2,3}
    \foreach \y in {0,2}
    {
    {
    \draw[thick] (\y,\x) -- (\y,\x+1);
    \draw[fill] (\y,\x) circle (0.1);
    \draw[fill] (\y,\x+1) circle (0.1);
    }
    }

     \draw[thick] (4,5) -- (4,6);
    \draw[fill] (4,5) circle (0.1);
    \draw[fill] (4,6) circle (0.1);
    \draw[thick] (4,5) -- (3,4);
    \draw[thick] (4,5) -- (5,4);
    
    \foreach \x in {1,2,3}
    {
    \draw[thick] (3,\x) -- (3,\x+1);
    \draw[fill] (3,\x) circle (0.1);
    \draw[fill] (3,\x+1) circle (0.1);
    }
    \draw[thick] (5,3) -- (5,4);
    \draw[fill] (5,3) circle (0.1);
    \draw[fill] (5,4) circle (0.1);

    \foreach \x in {3,4,5,6}
    {
    \foreach \y in {6,8}
    {
    \draw[fill] (\y,\x) circle (0.1);
    }
    }
    \draw[thick] (6,5) -- (8,5);
    \draw[thick] (6,3) -- (6,6);
    \draw[thick] (8,3) -- (8,6);

	\end{tikzpicture}
\end{center}
\caption{The three twin-free trees for $n=8$ satisfying $\imax(T)=8$.}\label{fig:extrTn8}
\end{figure}

\section{Further thoughts}\label{sec:conc}

F\"{u}redi~\cite[Conjecture~4.3]{Fur87} asked about the maximum of $\imax$ for other graph classes.
This was done e.g.~for triangle-free graphs in~\cite{HT93,CJ97}, connected unicyclic graphs~\cite{KGD08} and graphs with bounded degrees~\cite{Liu94}. Let $\max \imax(\G)$ denote the maximum number of maximal independent sets in a graph class $\G.$
This maximum grows exponentially in all cases mentioned, and $\lim_{n \to \infty} (\max \imax(\G))^{1/n}$ equals $\sqrt{2}$ and $\sqrt[3]{3}$ for trees and graphs respectively. 

In this paper, we proved that the minimum number of $i_{max}$ for twin-free graphs, bipartite graphs and trees is logarithmic, linear and exponential, respectively.
Due to this big difference in behaviour of the minimum of $\imax$, it may even be considered more interesting to study $\min \imax(\G)$ for some other graph classes $\G$ (restricted to twin-free graphs), since even the behaviour in terms of the order may be unclear.

\begin{question}
    What is the behaviour of $\min \imax(\G)$ for twin-free graphs in a graph class $\G$?
\end{question}

One plausible interesting direction can be to wonder what happens for $k$-partite graphs as $n \to \infty.$
For bipartite graphs, the bound was linear, while the extremal construction in Theorem~\ref{thr:imaxG} satisfies $\chi(G) \sim \log_2(n).$
Similarly, one can wonder about graphs with bounded clique number, with  Conjecture~\ref{conj:K3free} as a particular case.

Let $\nu(G)$ be the size of a maximum induced matching in $G$. Since every independent set can be extended to a maximal independent set and a set containing one end of each edge in an induced matching is an independent set, we conclude that $\imax(G) \ge 2^{\nu(G)}.$
 Note that $\nu(G) \ge \frac{Cn}{\Delta^2}$ (see \cite{Joos16}) grows linearly in $n$ for graphs with fixed maximum degree independent of $n$.
For such sparse graph classes (bounded maximum degree independent of $n$), $\min \imax(G)$ is exponential in terms of $n.$ 
Kahn and Park~\cite{KP19} proved that for hypercubes $Q_n$ of order $N=2^n$, we have $\lim_{n \to \infty} \imax(Q_n)^{1/N} = 2^{1/4}$, while also $\nu(Q_n) = \frac{N}{4}$ for every $n \ge 2$: in other words, $\imax(Q_n)$ is of the form $2^{(1+o(1))\nu(G)}$.
Note here that the hypercube $Q_n$ is a twin-free graph if $n \ge 3$.
For the twin-free trees in Figure~\ref{fig:construction_mod5}, we note that $\lim_{n \to \infty} \imax(T)^{1/n} = 3^{1/5}$ while $\nu(G)=\frac{n}{5}+1$ (when $5 \mid n$), so here $2^{ \nu(G)/n}$ is a lower bound for the limit that is not sharp.
Studying $\lim_{n \to \infty} (\min \imax(\G))^{1/n}$ for twin-free graphs in a graph class $\G$ might be interesting, especially if there is a natural graph class for which the constant is larger than $3^{1/5}$.
One natural graph class to consider are $r$-regular graphs for fixed $r$.
It is known that $\imax(C_n)=P(n),$ where $P(n)$ denotes the $n^{th}$ Perrin number, which is defined by $P(0)=3,P(1)=0, P(2)=2$ and the relation $P(n)=P(n-2)+P(n-3)$ for $n \ge 3.$
For $n \ge 3$, the quantity $P(n)^{1/n}$ is minimized by $n=4$, and the second smallest value is attained for $n=6$.
Since $C_4$ is not twin-free, for the class $\G$ of twin-free $2$-regular graphs, $\lim_{n \to \infty} (\min \imax(\G))^{1/n}=\imax(C_6)^{1/6} = 5^{1/6}$.

On the other side of the spectrum, when $G$ is $r$-regular with $r=n-1-t$, i.e.~$\overline{G}$ has small degree, the question may be interesting as well.
Such graphs satisfy $\frac{n}{t+1}\le \imax \le 2^t n.$
For $t=n^{1/s}$ and $\overline{G}$ being the Cartesian product of $s$ $K_t$s, we have for example that $\imax(G) \sim s n^{1-1/s}.$

%Other construction for $r = n/2-\Theta(n^{1/2})$ is bipartite complement of point-line incidence graph of projective plane. 

There may well be quite a number of other problems where a parameter $A$ can be unbounded in terms of another parameter $B$ if $A$ can increase when adding twins while $B$ does not.
This was for example also the obstruction in the related problem for $\nu(G)$, see~\cite{KPSX11}. 

One further direction is the relation between the vertex covering number $\tau(G)$ and $\imax(G).$
Hoang and Trung~\cite{HT19} proved that $\imax(G) \le 2^{\tau(G)}$ for every graph $G$. In the other direction there is no such bound, since the complete bipartite graph $K_{n,n}$ satisfies $\imax(G)=2$ and $\tau(G)=n.$
When restricting to twin-free graphs, we know that $\tau(G) \le n-1< 2^{\imax(G)}$ by Proposition~\ref{prop:imax<logn} and thus $\tau(G)$ is bounded by a function of $\imax(G)$. So the following question would tell us about the essential relationship between the two parameters.
\begin{question}
    What is the maximum possible vertex covering number $\tau$ of a twin-free graph with $\imax(G)=k$? 
\end{question}

Plausible candidates are $\tau(G)=O( \imax(G)^2 )$, and for given order $n$, $\tau(G)-\imax(G)\le \floorfrac{n}{2}-3.$
The first bound is attained by the complement of the Cartesian product $K_r \square K_r$ (one can connect every $K_r$ with an additional vertex), and the second one by the complement of a cycle of $K_3$s (i.e., a cycle $C_r$ for which every edge is connected to a new vertex).

\section*{Acknowledgement}

This project originated from the Mathematics Research Community workshop ``Trees in Many Contexts'', which was supported by the National Science Foundation under Grant Number DMS $1916439$. The authors would also like to thank Ferenc Bencs for suggesting a simplification in the proof of Theorem~\ref{thr:main}, and the referee for carefully reading the manuscript and suggesting several improvements and corrections.

\paragraph{Open access statement.} For the purpose of open access,
a CC BY public copyright license is applied
to any Author Accepted Manuscript (AAM)
arising from this submission.

\bibliographystyle{abbrv}
\bibliography{imax}

\end{document}